\documentclass[11pt]{amsart}
\usepackage{amssymb}
\usepackage[pdftex]{graphicx}
\usepackage[lofdepth,lotdepth]{subfig}
\DeclareCaptionLabelSeparator{nobreakspace}{\nobreakspace}
\newtheorem{theorem}{Theorem}[section]
\newtheorem{blank}{}[section]
\newtheorem{lemma}[theorem]{Lemma}
\newtheorem{corollary}[theorem]{Corollary}

\newtheorem{example}[theorem]{Example}
\newtheorem{definition}[theorem]{Definition}

\newcommand{\al}{{\alpha}}
\newcommand{\be}{{\beta}}

\newcommand{\ep}{{\varepsilon}}

\newcommand{\ga}{{\gamma}}
\newcommand{\si}{{\sigma}}
\newcommand{\om}{{\omega}}

\begin{document}

\title{On the Resolution Graph of a Plane Curve 
       }

\author{Jo\~ao Cabral \and Orlando Neto \and Pedro C. Silva}

\begin{abstract}
We show that the resolution graph of a plane curve singularity admits a canonical decomposition into elementary graphs.
\end{abstract}

\maketitle

\section{Introduction}
\noindent
The main purpose of this paper is to give a canonical decomposition of the resolution graph of an arbitrary plane curve that generalizes the decomposition of the resolution graph of an irreducible plane curve introduced by Brieskorn and Knorrer remark in \cite{BRIES}. The elementary graphs of this decomposition are of two types: \em smooth type \em (all branches are smooth, see Example \ref{EXE1}) and \em cusp type \em  (all branches are "cusps" and the graph is linear, see Example \ref{EXE2}).

\noindent
One can find in \cite{NESI1} a closed form description for the local fundamental group of a plane curve, given the Puiseux expansions of the branches of the curve or its Eggers tree. One of the motivations of this paper is to clear the way to give a closed form description of the local fundamental group of a plane curve from its equations (or its resolution graph) that can be implemented in a computer (\cite{paper}). 
Yang Jingen gave in  \cite{CHINES}  an interesting characterization of the resolution graphs. 

\noindent
Section \ref{INVSEC} recalls several topological invariants of plane curves and its behaviour by blow up. 
Theorem \ref{CONTR} shows that a normal crossings graph is a resolution graph if and only if it is contractible. 
Readers willing to accept this result may skip most of sections \ref{INVSEC} and \ref{COMBSEC}.
Sections \ref{TSMOOTH},  \ref{CONTINUOUS} and \ref{CSMOOTH} are dedicated to the characterization of the elementary graphs.
Section \ref{DECOM} proves the main result.

\section{Invariants of Plane Curve Singularities}\label{INVSEC}

\noindent
Let $Y$ be the germ of an irreducible plane curve defined by $f\in\mathbb C\{x,y\}$.
Assume that the tangent cone of $Y$ is transversal to $\{x=0\}$ and $Y$ has multiplicity $k$.
By the Puiseux Theorem there are a positive integer $n$ and
complex numbers $c_i$, $i\ge n$, such that $(n,k)=1$, $c_n\not=0$
and $Y$ admits the local parametrization 
\begin{equation}\label{PU1}
x=t^k,    \qquad    y=\sum_{i\ge n}c_it^i.  
\end{equation}
Consider the injective morphism from $\mathbb C\{x\}$ to $\mathbb C\{t\}$ that takes $x$ into $t^k$.
We will sometimes denote $t$ by $x^{\frac{1}{k}}$. For $\alpha \in\mathbb Q_{+}$, set $a_\al=c_i$ if $\alpha=i/k$ and $a_\al=0$ if $k\al\not\in\mathbb Z$.
We say that (\ref{PU1}) and $y=\sum_\al a_\al x^\al$
 are \em Puiseux expansions \em of $Y$ relatively to the system of local coordinates $(x,y)$.
 
\noindent
There is a certain ambiguity on the determination of the coefficients of the Puiseux expansion of a plane curve.
When we compare two Puiseux expansions 
\begin{equation}\label{TWOPU}
y=\sum_\al a_\al x^\al ~ \mbox{ and } ~ y=\sum_\al b_\al x^\al,
\end{equation}
we shall assume that $a_\al=b_\al$ for $\al\le \be$ if $\sum_{\al\le \be} a_\al x^\al$ and $\sum_{\al\le \be} b_\al x^\al$ 
are parametrizations of the same plane curve  (see Section  4.1 of \cite{CTCW}).

\noindent
If $Y$ admits the Puiseux expansion $y=\sum_\al a_{\al} x^\al$, there are unique rational numbers $\alpha_1<\cdots<\alpha_g$ such that:
\begin{enumerate}
\item For $1\le k\le g$, $a_{\al_k}\not=0$.
\item Set $\al_{g+1}=+\infty$. 
If $a_\al\not = 0$ and $\al<\al_{k+1}$, $\al$ is a linear combination with integer coefficients of $\al_1,\ldots,\al_k$.
\item $\alpha_{i}$, $i\in\{2,\ldots,g\}$, is not a linear combination with integer coefficients of $\alpha_1,\ldots,\alpha_{i-1}$.
\end{enumerate}
The rational numbers $\alpha_1,\ldots, \alpha_g$ are called the \em Puiseux exponents \em of $Y$.

\noindent
Given two plane curves $Y,Z$ with Puiseux expansions (\ref{TWOPU}), 
we call \em contact exponent \em of $Y$ and $Z$ to the smallest $\gamma$ such that $a_\ga\not=b_\ga$. 
We will denote it by $o(Y,Z)$.

\noindent
Let $Y,Z$ be two plane curves with Puiseux expansions (\ref{TWOPU}). Given a positive rational $\gamma$, $o(Y,Z)=\gamma$ if and only if 
$y=\sum_{\al \leq \beta} a_\al x^\al$ and $y=\sum_{\al \leq \beta} b_\al x^\al$ define the same curve for all positive rational $\beta<\gamma$ but $y=\sum_{\al\leq \gamma} a_\al x^\al$ and $y=\sum_{\al\leq \gamma} b_\al x^\al$ define different curves.

\noindent
Notice that $o(Y,Z)=1$ if and only if $Y$ and $Z$ have different tangent cones.

\begin{theorem}\label{BLOWINV1}
Let $\widetilde{Y}$ be the strict transform of an irreducible plane curve $Y$.
Let $E$ be the exceptional divisor of the blow up.
Assume that $Y$ has Puiseux exponents $\al_1,\ldots,\al_g$. 
\begin{enumerate}
\item If $\al_1\ge 2$, $\widetilde{Y}$ is transversal to the exceptional divisor and has Puiseux exponents $\al_1-1,\al_2-1,\ldots,\al_g-1$. 
\item Assume that $\al_1<2$. Then $\widetilde{Y}$ is tangent to the exceptional divisor and $o(\widetilde{Y},E) =1/(\al_1-1)$.
\begin{enumerate}
\item If $1/(\al_1-1)$ is not an integer,
the Puiseux exponents of $\widetilde{Y}$ are 

\noindent
$1/(\al_1-1),\al_2/(\al_1-1)-1,\ldots,\al_g/(\al_1-1)-1$.
\item If $1/(\al_1-1)$ is  an integer,
the Puiseuxs exponents of $\widetilde{Y}$ are $\al_2/(\al_1-1)-1,\ldots,\al_g/(\al_1-1)-1$.
\end{enumerate}
\end{enumerate}
\end{theorem}
\begin{proof}
It follows from Theorem 3.5.5 of \cite{CTCW}.
\end{proof}

\begin{corollary}\label{BLOWDINV1}
Let $X$ be a smooth complex surface.
Let $E$ be an exceptional divisor of $X$.
Let $\pi : X\to\underline{X}$ be the blow down of $X$ along $E$.
Let $Y$ be the germ of a irreducible plane curve at a point $\sigma$ of $E$.
Assume that $Y$ has Puiseux exponents $\al_1,\ldots,\al_g$.
\begin{enumerate}
\item If $Y$ is transversal to $E$, $\pi(Y)$ has Puiseux exponents $\al_1+1,\al_2+1\ldots,\al_g+1$.
\item Assume that $Y$ is tangent to $E$.
\begin{enumerate}
\item If $o(Y,E)$ is not an integer, $o(Y,E)=\al_1$ and $\pi (Y)$ has Puiseux exponents 
$1+1/\al_1,(\al_2+1)/\al_1,\ldots,(\al_g+1)/\al_1$.

\item If $o(Y,E)$ is an integer, $\pi (Y)$ has Puiseux exponents 

$1+1/o(Y,E),(\al_1+1)/o(Y,E),\ldots,(\al_g+1)/o(Y,E)$.
\end{enumerate}
\end{enumerate}
\end{corollary}

\begin{theorem}\label{BLOWINV2}
Let $Y,Z$ be irreducible plane curves with the same tangent cone.
Let $\widetilde{Y},\widetilde{Z}$ be their strict transforms and $E$ the exceptional divisor.
Let $\mu$ $[\nu,\widetilde{\mu},\widetilde{\nu}]$ be the first Puiseux exponent of $Y$ $[Z,\widetilde{Y},\widetilde{Z}]$.
Assume that $\nu\le \mu$.
\begin{enumerate}

\item If $\nu\ge 2$, $o(\widetilde{Y},\widetilde{Z})=o({Y},{Z})-1$ and $\widetilde{Y},\widetilde{Z}$ are transversal to $E$. 

\item If $\nu<2$ and $\mu>2$, $o(\widetilde{Y},\widetilde{Z}) =1$ and $\widetilde{Z}$ is tangent to $E$.

\item Assume that $\mu<2$. 
Then $\widetilde{Y},\widetilde{Z}$ are tangent to $E$.
Moreover, $\widetilde{\mu}\le\widetilde{\nu}$.

\begin{enumerate}

\item Assume that $o(Y,Z)>\mu$. Then $\mu=\nu$, $\widetilde{\mu}=\widetilde{\nu}$ and 

$o(\widetilde{Y},E)=o(\widetilde{Z},E)=\widetilde{\mu}<o({Y},{Z})/(\mu-1)-1=
o(\widetilde{Y},\widetilde{Z})$.

\item Assume that $o(Y,Z)=\mu$. Then $\mu=\nu$ and 

$o(\widetilde{Y},E)=o(\widetilde{Z},E)=\widetilde{\mu}=\widetilde{\nu}=
o(\widetilde{Y},\widetilde{Z})$.

\item Assume that $o(Y,Z)<\mu$. Then $o(Y,Z)=\nu$  and

$o(\widetilde{Y},E)=o(\widetilde{Y},\widetilde{Z})=\widetilde{\mu}<\widetilde{\nu}=o(\widetilde{Z},E)$.

\end{enumerate}
\end{enumerate}
\end{theorem}

\begin{proof}
Statements $(1)$ and $(2)$ of Theorem \ref{BLOWINV1} imply statement $(2)$. Let (\ref{PU1}) be a parametrization of $Y$. Therefore the curve $\widetilde{Y}$ has a parametrization given by
\begin{equation}\label{PUBLOW}
x=t^k,    \qquad    \frac{y}{x}=\sum_{i\ge n}c_it^{i-k}=t^{n-k}\gamma(t),\ \gamma(0)\neq 0.  
\end{equation}
Assume $\nu\geq 2$. Then the parametrization (\ref{PUBLOW}) is a Puiseux expansion and statement $(1)$ holds. 
Assume $\mu<2$. In this case, the parametrization (\ref{PUBLOW}) is not a Puiseux expansion. Let $x=t^k,\ y=t^{n}\overline{\gamma}(t),\overline{\gamma}(0)\neq 0$ be a Puiseux expansion of $Z$. Let $\alpha(t)=\sum_{i\geq 0}\alpha_i t^{i},\overline{\alpha}(t)=\sum_{i\geq 0}\overline{\alpha}_i t^{i}$ such that $\alpha(t)^{n-k}=\gamma(t),\overline{\alpha}(t)^{n-k}=\overline{\gamma}(t)$. Let $\beta(t)=\sum_{i\geq 0}\beta_i t^{i},\overline{\beta}(t)=\sum_{i\geq 0}\overline{\beta}_i t^{i}$, such that $t=w\beta(w)$ is a solution of $w=t\alpha(t)$ and $t=w\overline{\beta}(w)$ is a solution of $w=t\overline{\alpha}(t)$. To prove statement $(3)$ it suffices to show that if $\alpha(t)\equiv\overline{\alpha}(t)$ \em mod \em $(t)^{i}$, $i\in\mathbb N$, $\beta(t)\equiv\overline{\beta}(t)$ \em mod \em $(t)^{i}$. Also, if $\alpha_{i+1}\neq\overline{\alpha}_{i+1}$, $\beta_{i+1}\neq\overline{\beta}_{i+1}$.

\noindent We can check easily that $\alpha_i$ depends only on $c_n,\ldots,c_{n+i}$, for all $i\geq 0$, and $\alpha(0)\neq 0$. From the equality $w=t\alpha(t)$ we get
\begin{equation}\label{REPAR1} 
w=\sum_{i\geq 0}\sum_{j_0,\ldots,j_i}\alpha_i\beta_{j_{0}}\ldots\beta_{j_{i}}w^{j_{0}+\cdots+j_{i}+i+1}.
\end{equation}
From equality (\ref{REPAR1}) we deduce the equations
\begin{equation}\label{BETA1}
\alpha_{0}\beta_{0}=1,
\end{equation}
\begin{equation}\label{BETAK}
\alpha_{0}\beta_{k-1}+\sum_{i=1}^{k-1}\sum_{j_0+\cdots+j_i=k-1-i}\alpha_i\beta_{j_{0}}\ldots\beta_{j_{i}}=0,\ k\in\mathbb N.
\end{equation}
Therefore $\beta(0)\neq 0$ and the parametrization 
\begin{displaymath}
x=(w\beta(w))^{\frac{k}{(k,n-k)}},\;\textstyle{\frac{y}{x}}=w^{\frac{n-k}{(k,n-k)}}
\end{displaymath}
is a Puiseux expansion.
Also, by induction, we prove that $\beta_i$ depends only on $\alpha_0,\ldots,\alpha_i$, for all $i\geq 0$.
\end{proof}

\begin{corollary}\label{BLOWDINV2}
Let $X$ be a smooth complex surface.
Let $E$ be an exceptional divisor of $X$.
Let $\pi : X\to\underline{X}$ be the blow down of $X$ along $E$.
Let $Y,Z$ be two germs of  irreducible plane curves at a point $\sigma$ of $E$.
The curves $\pi(Y),\pi(Z)$ are germs of irreducible plane curves at the point $\sigma(E)$ of $\underline{X}$,
with the same tangent cone.
Let $\mu$ $[\nu]$ be the first Puiseux exponent of $Y$ $[Z]$.
Assume that $\mu\le\nu$.
\begin{enumerate}
\item If $Y,Z$ are transversal to $E$, $o(\pi(Y),\pi(Z))=o(Y,Z)+1$.
\item If $Y$ is transversal to $E$ and $Z$ is tangent to $E$, $o(\pi(Y),\pi(Z))=1+1/\nu$.
\item Assume that $Y$ and $Z$ are tangent to $E$.
\begin{enumerate}
\item If $o(Y,E)=o(Z,E)<o(Y,Z)$, $o(\pi(Y),\pi(Z))=(\mu-1)(o(Y,Z)+1)$.
\item If $o(Y,E)=o(Z,E)=o(Y,Z)$, $o(\pi(Y),\pi(Z))=1+1/o(Y,Z)$.
\item If $o(Y,E)=o(Y,Z)<o(Z,E)$, $o(\pi(Y),\pi(Z))=1+1/o(Y,Z)$.
\end{enumerate}
\end{enumerate}
\end{corollary}

\section{The Combinatorics of the Resolution Algorithm}\label{COMBSEC} 

\begin{blank}\label{PARAGRAFO}\em
Let $Y$ be the germ at the origin of a plane curve defined by $f\in\mathbb C\{x,y\}$.
Let $X$ be a polydisc where $f$ converges.
We can assume that $Y$ has no singular points on $X\setminus \{(0,0)\}$.
Set $X_0=X$, $Y_0=Y$, $\si_0=(0,0)$.
Let $\pi_1:X_1\to X_0$ be the blow up of $X$ with center $\si_0$.
Set $D_1=\pi_{1}^{-1}(\si_0)$.
Let $Y_1$ be the strict transform of $Y$.

\noindent
Let us define recursively a sequence of triples $(X_\ell,D_\ell,Y^{(\ell)})$, where 
$X_\ell$ is a smooth complex surface, $D_\ell$ is a divisor with normal crossings of $X_\ell$ 
and $Y^{(\ell)}$ is a union of germs of plane curves at points of $D_\ell$.
Let $W_\ell$ be set of points of $D_\ell\cap Y_\ell$ where $D_\ell\cup Y^{(\ell)}$ is not a divisor with normal crossings.
Let $\pi_{\ell+1}:X_{\ell+1}\to X_\ell$ be the blow up of $X_\ell$ with center $C_\ell$.
Let $Y^{(\ell+1)}$ be the strict transform of $Y^{(\ell)}$ by $\pi_{\ell+1}$.
Set $W_{\ell+1}=\pi^{-1}(D_\ell)$.
After a finite number of steps $D_\ell\cup Y^{(\ell)}$ is a normal crossings divisor, $W_\ell$ is the empty set and $X_{\ell+1}=X_\ell$.
\end{blank}

\noindent
Let us describe this procedure in purely combinatorial terms.

\begin{definition}\label{STEPDEF}
Consider the following data: 
\begin{enumerate}
\item A weighted tree $\mathcal D$. The vertices of $\mathcal D$ are called \em divisors. \em 
The weight $\om_E$  of a divisor $E$ of $\mathcal D$ is a negative integer. 
We call $\om_E$ the \em  self-intersection \em number of $E$.
If two divisors are connected by an edge, we say that they \em intersect each other. \em 
\item A nonempty set $o_{\mathcal D}$. The elements of $o_{\mathcal D}$ are called \em points \em of $\mathcal D$.
If $\mathcal D$ is the empty graph, $o_{\mathcal D}$ has only one element. 
We associate to each point $\sigma$ of $\mathcal D$ a set of divisors $\mathcal D_\sigma$ of $\mathcal D$.
If $E\in\mathcal D_\sigma$, we say that $\sigma$ \em belongs \em to $E$ and $E$ \em contains \em $\sigma$.

\item A family of disjoint finite sets $\mathcal C_{\sigma}$, where $\sigma$ is a point of $\mathcal D$. 
The elements of  $\mathcal C_{\sigma}$ are called \em branches. \em 
We associate to each branch $i$ of $\mathcal C_{\sigma}$ a nonnegative integer $g_i$ and rational numbers $\ep_{i,j}$, $1\le j\le g_i$.
We call the rational numbers $\ep_{i,j}$ the \em Puiseux exponents \em of the branch $i$.
Let $\widetilde{\mathcal C}_\sigma$ be the union of $\mathcal C_{\sigma}$ with the set of divisors of $\mathcal D$ that contain $\sigma$.
We define a map $o:\{(i,j):i,j\in \widetilde{\mathcal C}_\sigma,\: i\neq j\}\rightarrow\mathbb Q$. 
We call $o(i,j)$ the \em contact order \em of $i$ and $j$.
\end{enumerate}
We call the set of data (1),(2),(3) a \em resolution step \em if it verifies the conditions: 
\begin{enumerate}
\item A point of $\mathcal D$ belongs to at most two divisors. Given two divisors, there is at most one point that belongs to both.
\item If a point of $\mathcal D$ belongs to two divisors, they intersect each other.
\item If the graph $\mathcal D$ is nonempty, each point of $\mathcal D$ belongs to some divisor.
\item $o(i,j)\ge 1$ for each $i,j\in\widetilde{\mathcal C}_\sigma$.
\item $o(i,j)=o(j,i)$ for each $i,j\in \widetilde{\mathcal C}_\sigma$.
\item $o(i,k)\ge$min$\{o(i,j),o(j,k)\}$ for each $i,j,k\in\widetilde{\mathcal C}_\sigma$.
\item $1=\ep_{i,0}<\ep_{i,1}<\cdots<\ep_{i,g_i}$ for each $i\in \mathcal C_\sigma$.
\item $\ep_{i,j}$ is not a linear combination with integer coefficients of $\ep_{i,1},\ldots,\ep_{i,j-1}$ for $1\le j\le g_i$.
\end{enumerate}
We say that a resolution step is a \em plane curve \em if the graph $\mathcal D$ is the empty graph. 
We say that a resolution step has \em normal crossings \em at a point $\sigma$ of $\mathcal D$ if
the point $\sigma$ belongs to exactly one divisor $E$, 
the point $\sigma$ has exactly one branch $Y$and $o(Y,E)=1$.
We say that a resolution step is a \em normal crossings graph \em if it is normal crossings at each point $\sigma$ of $o_{\mathcal D}$.
We will represent a normal crossings graph in the following way. We will connect by an edge the asterisc that represents a branch of $Y$ to the vertex that represents a divisor if there is a point that belongs to both.

\noindent
Let $\sigma$ be a point of $\mathcal D$.
Consider in $\mathcal C_\sigma$ the equivalence relation $i \sim j$ if $o(i,j)>1$.
We call an equivalence class $\ell$ of  $\mathcal C_\sigma$ a \em tangent line \em of $\sigma$.
\end{definition}

\noindent
Let $X$ be a smooth complex surface. 
Let $D$ be a normal crossings divisor of $X$.
We will call \em  divisors \em of $D$ to the irreducible components of $D$.
Let $Y$ be the union of a finite set of germs of plane curves of $X$ at points of $D$.
If $D$ is not empty, $Y=\cup_{\si\in A}Y_\si$, where $A$ is a finite subset of $D$
 and $Y_\si$ is the germ at $\sigma$ of a plane curve with no irreducible component contained in $D$.
If $D$ is empty, $Y$ is the germ of a plane curve at a point $\si_0$ of $X$.

\noindent
Let us associate to a triple $(X,D,Y)$ a resolution step denoted by $[X,D,Y]$.
Let $\mathcal D$ be the graph with vertices the divisors of $D$.
Two vertices of $\mathcal D$ are connected if they intersect each other.
We will label each divisor $E$ of $D$ with its self-intersection number $\om_E$ (see \cite{SELFINTERS}).
If $D$ is not empty, set $o_{\mathcal D}=A$.
Otherwise set $o_{\mathcal D}=\{ \si_0 \}$.
Let $\mathcal D_\si$ be the set of divisors of $D$ that contain $\si$, for each $\si\in o_{\mathcal D}$.
Let $Y_i$, $i\in\mathcal C_\si$, be the irreducible components of $Y_\si$.
Let $\ep_{ij}$, $1\le j\le g_i$, be the Puiseux exponents of $Y_i$ for each $i\in C_\si$. 
If $i,j\in\mathcal C_\si$ and $i\not=j$, set $o(i,j)=o(Y_i,Y_j)$.
If $i\in\mathcal C_\si$ and $E\in \mathcal D_\si$, set $o(i,E)=o(Y_i,E)$.
If $E,F\in \mathcal D_\si$ and $E\not=F$, set $o(E,F)=1$.

\noindent
Notice that $[X,D,Y]$ is normal crossings if and only if $D\cup Y$ is a normal crossings divisor.
Moreover, $[X,D,Y]$ is a \em first step \em if and only if $D$ is empty.

\begin{definition}\label{BU}
Let us define the \em blow up \em $(\overline{\mathcal D},(\overline{\mathcal C}_\si))$ of a resolution step $(\mathcal D,(\mathcal C_\si))$.
Let $c_{\mathcal D}$ be the set of $\si\in o_{\mathcal D}$ such that 
$(\mathcal D,(\mathcal C_\si))$ is not normal crossings at $\si$.
The vertices of $\overline{\mathcal D}$ will be the vertices of $\mathcal D$ and the elements of $c_{\mathcal D}$.
For each divisor $E$ of $\overline{\mathcal D}\setminus c_{\mathcal D}$ we set $\overline{\om}_E=\om_E-k$, where $k$ is the number of elements of $c_{\mathcal D}$ that belong to $E$.
We set $\overline{\om}_E=-1$ for each element $E$ of $c_{\mathcal D}$.
Let $\si\in c_{\mathcal D}$.
If $\si$ belongs to a divisor $E$, we connect $\si$ to $E$ by an edge.
If $\si$ belongs to a pair of divisors $E,F$, we withdraw the edge that connects $E$ to $F$ and set
\[
o_{\overline{\mathcal D}}=(o_{\mathcal D}\setminus c_{\mathcal D})\cup\cup_{\si\in c_{\mathcal D}}\{\mbox{ tangent lines of }\si\}.
\]
We set $\overline{\mathcal C}_\si=\mathcal C_\si$ for each $\si\in o_{\mathcal D}\setminus c_{\mathcal D}.$
We set $\overline{\mathcal C}_\ell=\ell$ for each tangent line $\ell$ of $\si\in c_{\mathcal D}$.
We define the Puiseux exponents and the contact exponents of $(\overline{\mathcal D},(\overline{\mathcal C}))$ according to Theorems \ref{BLOWINV1} and \ref{BLOWINV2}.
\end{definition}

\begin{definition}\label{BD} 
Let us define the \em blow down \em $(\underline{\mathcal D},(\underline{\mathcal C}_\si))$ 
of the resolution step $(\mathcal D,(\mathcal C_\si))$.
The vertices of $\underline{\mathcal D}$ are the vertices $E$ of $\mathcal D$ such that $\om_E\not=-1$.
For each vertex $E$ of $\underline{\mathcal D}$ we set $\underline{\om}_E=\om_E+k$, where $k$ is the number of divisors $F\in {\mathcal D}\setminus\underline{\mathcal D} $ such that $F$ intersects $E$.
If two divisors $E,F$ of $\underline{\mathcal D}$ intersect a divisor of weight $-1$ of $\mathcal D$, we connect $E$ and $F$ by an edge. 
Let $o^*_{\mathcal D}$ be the set of points $\si\in o_{\mathcal D}$ such that $\si$ does not belong to a divisor of weight $-1$.
We set $o_{\underline{\mathcal D}}=o^*_{\mathcal D}\cup ({\mathcal D}\setminus\underline{\mathcal D})$.
If $\si\in o^*_{\mathcal D}$, set $\underline{\mathcal C}_\si={\mathcal C}_\si$. 
If $E\in {\mathcal D}\setminus\underline{\mathcal D}$, i.e. $\omega_{E}=-1$, set $\underline{\mathcal C}_E=\cup\{{\mathcal C}_\si$: $\si$ belongs to $E\}$.
We define the Puiseux exponents and the contact exponents of $(\underline{\mathcal D},(\underline{\mathcal C}_\si))$ 
according to Corollaries \ref{BLOWDINV1} and \ref{BLOWDINV2}. 

\end{definition}

\noindent
The following result is an  immediate consequence of the previous definitions.

\begin{lemma}\label{UPDOWNLEMMMA}
The blow up of the resolution step of $(X,D,Y)$ equals the resolution step of the blow up of $(X,D,Y)$.
The blow down of the resolution step of $(X,D,Y)$ equals the resolution step of the blow down of $(X,D,Y)$.
\end{lemma}

\begin{definition}\label{RESOLG} 
We say that a normal crossings graph is a \em resolution graph \em if 
there is a plane curve that is  transformed  into the normal crossings graph by a sequence of blow ups.
We say that a normal crossings graph is \em contractible \em 
if there is a sequence of blow downs that takes the normal crossings graph into a plane curve.
On the step previous to the last blow down, we get the empty graph and the normal crossings graph reduces to a single vertex.
We call this vertex the \em root \em of the normal crossings graph.
\end{definition}

\noindent
The root of a resolution graph is therefore the divisor created by the first blow up.
We call \em root of a resolution step \em to the strict transform of the divisor created by the first blow up.

\begin{theorem}\label{CONTR}
A normal crossings graph is a  resolution graph  if and only if it is contractible.
\end{theorem}

\begin{proof}
If the normal crossings  graph is contractible, we can blow it down  into a plane curve and reverse the procedure, 
obtaining the initial resolution graph.
\end{proof}

\noindent
A resolution graph $\mathcal R$ is \em minimal \em if there is no divisor $E$ of weigth $-1$ of $\mathcal R$ such that after blowing down $E$ we obtain a normal crossings graph. Given the germ of a plane curve $Y$, there is one and only one minimal resolution graph $\mathcal R_Y$ that is a resolution graph of $Y$.

\noindent
Given a divisor $E$ of a resolution step we call \em valence \em of $E$ to the number of divisors of that resolution step that intersect $E$.
 We say that a resolution step is \em linear \em if all of its divisors have valence $0,1$ or $2$.
If a resolution step has only one divisor, this divisor is simultaneously the root and a terminal vertex of the resolution step.   
 Otherwise we call \em terminal vertex \em of a resolution step to its vertices of valence $1$ that are not the root.

\section{Resolution Graphs of Smooth Type} \label{TSMOOTH}

\noindent
Let $Y$ $[\Sigma]$ be a [smooth] plane curve of $(X,o)$. We call the pair $(Y,\Sigma)$ a \em logarithmic plane curve \em if $Y$ and $\Sigma$ have no common irreducible components. We say that a logarithmic plane curve is \em singular \em  if the curve $Y\cup \Sigma$ is not a normal crossings divisor. We call \em branches \em of $(Y,\Sigma)$  to the branches of $Y$. Let $C$ be a smooth curve transversal to $\Sigma$. Define $\tau_{C}(Y)$ as the smallest positive integer $\tau$ such that $o(Y_i,C)<\tau$ for each singular branch $Y_i$ of $Y$.

\noindent
Let $\pi:\widetilde{X}\rightarrow X$ be the composition of blow ups that desingularizes $Y\cup \Sigma$. Let $\widetilde{Y}$ [$\widetilde{\Sigma}$] be the strict transform of $Y$ [$\Sigma$] by $\pi$. We call \em resolution graph \em of $(Y,\Sigma)$  to the graph $\mathcal R^{\Sigma}_{Y}$ obtained in the following manner:
the vertices of $\mathcal R^{\Sigma}_{Y}$ are the vertices of  $\mathcal R_{Y\cup\Sigma}$ plus $\widetilde{\Sigma}$;
the branches of $\widetilde Y$ are the branches of $\widetilde{Y\cup\Sigma}$ minus $\widetilde\Sigma$.
We call $\widetilde\Sigma$ the \em root \em of $\mathcal R^{\Sigma}_{Y}$.
We set $\om_{\widetilde \Sigma}=-c-1$, where $c$ is the number of times we blow up a point of a strict transform of $\Sigma$.

\noindent
We call \em log resolution graph \em to a pair $\mathcal R^{\Sigma}=(\mathcal R,\Sigma)$, where:

\noindent
$(i)$ $\mathcal R$ is a normal crossings graph;

\noindent
$(ii)$ $\Sigma$ is a divisor of valence $1$ of $\mathcal R$;

\noindent
$(iii)$ If we consider $\Sigma$ a branch of $\mathcal R$, $\mathcal R$ becomes the resolution graph of a plane curve;

\noindent
$(iv)$ The image $\underline \Sigma$ of $\Sigma$, by the blow down of $\mathcal R$ into a plane curve, is a smooth component of $Y$;

\noindent
$(v)$ $\om_\Sigma=-c-1$, where $c$ is the number of times necessary to blow down a divisor intersecting $\Sigma$ when we blow down $\mathcal R$ into a plane curve.

\noindent
A normal crossings graph $\mathcal R$ is a log resolution graph if and only if there is a logarithmic plane curve $(Y,\Sigma)$ such that $\mathcal R=\mathcal R^{\Sigma}_{Y}$.

\noindent
Let $\mathcal D$ be a resolution step and $E$ a smooth divisor of $\mathcal D$ 
with weight $\omega_E$ and valence $\vartheta_E$. 
We call \em rectified weight of E \em to the number $\omega_E+\vartheta_E$.

\begin{definition}
Let $\mathcal R$ [$\mathcal R^{\Sigma}$] be a [log] resolution graph. 
We say that  $\mathcal R$ [$\mathcal R^{\Sigma}$] is a [\em log\/\em ]\em  elementary graph of smooth type \em
if the root of $\mathcal R$ [$\mathcal R^{\Sigma}$] has rectified weigth $-1$ and the other divisors of  $\mathcal R$ [$\mathcal R^{\Sigma}$] have  rectified weigth $0$.
\end{definition}

\begin{theorem}\label{SMOOTHLOG}
Let $(Y,\Sigma)$ be a logarithmic plane curve. Assume $\Sigma$ smooth. 
The graph $\mathcal R^{\Sigma}_{Y}$ is a log elementary graph of smooth type if and only if 
all branches of $Y$ are non singular and transversal to $\Sigma$.
\end{theorem}

\begin{proof}
We will use the notations of Paragraph \ref{PARAGRAFO}. Set $\Sigma^{(0)}=\Sigma$. Let $\Sigma^{(\ell+1)}$ be the strict transform of $\Sigma^{(\ell)}$ by $\pi_{\ell+1}$, for each $\ell\geq 0$.
Assume that all branches of $Y$ are non singular and transversal to $\Sigma$. Set $D_1=\pi^{-1}_{1}(\Sigma)$.
Consider the statements: \em 

\noindent
$(1)_\ell$ $D_\ell\cup Y^{(\ell)}$ is normal crossings at each point of a non terminal divisor $E$.

\noindent
$(2)_\ell$  $\Sigma^{(\ell)}$ has rectified weight $-1$ and the remaining  irreducible components of $D_{\ell}$ have rectified weight $0$.

\noindent
$(3)_\ell$ Each branch of $Y^{(\ell)}$ is non singular, intersects only one divisor and is transversal to this divisor.

\noindent\em
Statements $(i)_\ell$, $i=1,2,3$, $\ell\ge 1$, can be proved by induction in $\ell$.
Let $N$ be the maximal element of $\{o(Y_i,Y_j): i,j\in I, i\not=j\}$.
The curve $Y$ is desingularized after $N$ blow ups.
By $(2)_N$, $\mathcal R^{\Sigma}_{Y}$ is a log elementary graph of smooth type.

\noindent
Let us prove the converse.
We associate to each vertex $E$  of a rooted tree with root $\Sigma$ its \em depth, \em the number of edges we have to cross in order to go from  $E$ to $\Sigma$.
We call depth of the tree to the supreme of the depths of its elements. 
Let $N$ be the depth of $\mathcal R^{\Sigma}_{Y}$.
Let $\pi:\widetilde X\to X$ be the minimal sequence of blow ups that desingularizes $(Y,\Sigma)$.
Set $D_N=\pi^{-1}(\Sigma)$. Let $Y^{(N)}$ [$\Sigma^{(N)}$] be the strict transform of $Y$ [$\Sigma$] by $\pi$.
Hence $Y^{(N)},\Sigma^{(N)},D_N$ verify $(i)_N$, $1\le i\le 3$. The divisors of depth $N$ have valence $1$, hence weigth $-1$. Let us blow down these divisors.
Let $Y^{(N-1)}, \Sigma^{(N-1)}, D_{N-1}$ 
be the images of $Y^{(N)}, \Sigma^{(N)},D_{N}$ by the blow down.
Statements $(1)_{N-1},(2)_{N-1}$ are easily verified. 
If a branch of $Y^{(N-1)}$ was singular, its strict transform would be singular or tangent to a divisor.
If a branch of $Y^{(N-1)}$ intersected two divisors, the blow up of its intersection would create a divisor of positive rectified weigth.
If a branch of $Y^{(N-1)}$ was tangent to a divisor, its strict transform would intersect two divisors.
Hence we can iterate the procedure. 
By $(3)_0$ the branches of $Y$ are smooth and transversal to $\Sigma$.
\end{proof}

\begin{corollary}\label{SMOOTH}
Let $Y$ be a singular plane curve.
Its resolution graph $\mathcal R_Y$ is an elementary graph of smooth type if and only if 
all branches of $Y$ are non singular.
\end{corollary}

\begin{proof}
Let $Y$ be a singular plane curve. Let $\Sigma$ be a smooth curve transversal to $Y$. The graph $\mathcal R_Y$ is an elementary graph of smooth type if and only if  $\mathcal R^{\Sigma}_{Y}$ is a log elementary graph of smooth type.
\end{proof}

\noindent 
We shall depict the vertices of an elementary graph of smooth type using white dots.

\begin{example}\label{EXE1}
\em An elementary graph of smooth type. \em
Let $Y_i=\{ y=\varphi_i (x)\}$, $1\le i\le 5$, where 
$\varphi_1(x)=x^{2}$,
$\varphi_2(x)=x^{3}+x^{4}$, 
$\varphi_3(x)=x^{3}-x^{4}$, 
$\varphi_4(x)=-x^{3}+x^{4}$ and 
$\varphi_5(x)=-x^{3}-x^{4}$. Set $Y=\cup_{i=1}^{5}Y_i$.

\begin{figure}[h!]
\centering
\includegraphics{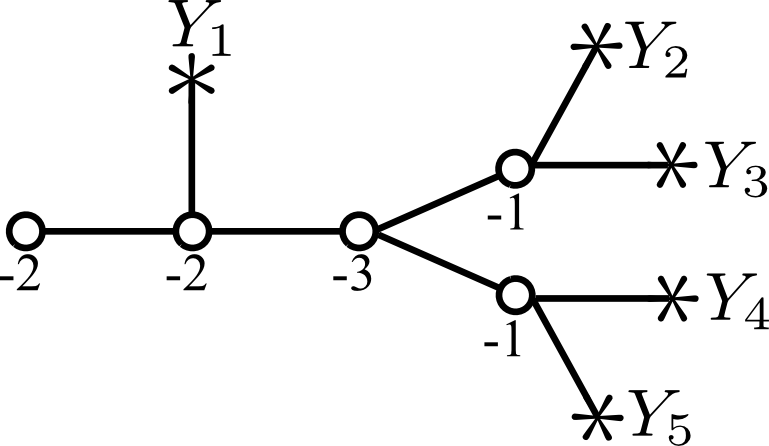}
\label{FIGEXE1}
\caption{The elementary graph $\mathcal R_{Y}$.}
\end{figure}

\end{example}

\section{Geometry of Continuous Fractions}\label{CONTINUOUS}  

\noindent
Let $(X,o)$ be a germ of a smooth complex manifold. Let $\Sigma$ and $C$ be two smooth transversal curves of $(X,o)$. 
Choose a system of local coordinates $(x,y)$ of  $(X,o)$ such that $\Sigma=\{x=0\}$ and $C=\{ y=0\}$.
Let $\Xi$ be the union of $\{\Sigma,C\}$ with the set of divisors that are created when we desingularize the curves $\{y^k-x^n=0\}$, with $(k,n)=1$.
We will identify each element of $\Xi$ with each of its strict transforms.

\noindent
Let $\mathbb Q_+$ be the set of positive integers. Set  $\overline{\mathbb Q}_{+}=\mathbb Q_{+}\cup \{0,+\infty\}$.
The main purpose of this section is to prove Lemma \ref{BIJECCAO}.

\noindent
Let us give an invariant definition of $\Xi$.
Set $X_0=X$, $D_0=\Sigma\cup C$. 
Let $\pi_{n}:X_{n}\rightarrow X_{n-1}$ be the blow up of $X_{n-1}$ along the singular locus of $D_{n-1}$. 
Set $D_{n}=\pi^{-1}_{n}(D_{n-1})$. 
Let $\mathcal R_{n}$ be the dual graph of $D_n$.
We will identify each irreducible component of each $D_n$ with the corresponding strict transforms by blow up. 
Let $\Xi_n$ be the set of  equivalence classes of the vertices of  $\mathcal R_{n}$. 
Set $\Xi=\cup_n\Xi_n$.
If $E,F\in \Xi$ have representatives $E',F'$ such that $E'\cap F'\not=\emptyset$, we will denote by $E \ast F$ the equivalence class of the  divisor created by the blow up of $E'\cap F'$.
We will identify an element of $\Xi$ with a convenient representative whenever this identification does not create a serious ambiguity.

\noindent
Let $E,F,G\in\Xi$. There is $m$ such that $E,F,G$ are vertices of $\mathcal R_{m}$.
We say that $F$ is \em between $E$ and $G$ \em, and write $E-F-G$, if $F$ is a vertice of each path of $\mathcal R_{m}$ that connects $E$ and $G$.
Since the graphs $\mathcal R_{m}$ are linear, the relation $E<F$ if 
$\Sigma - E - F$ is a total order of $\Xi$.

\noindent
Given integers $a_1,\ldots,a_n$, $a_1\geq 0$, $a_2,\ldots,a_n\geq 1$, $n\geq 1$, we define recursively the rational number  $[a_1,\ldots,a_n]$ by $[a_1]=a_1$ and $[a_1,\ldots,a_n]=a_1+1/[a_2,\ldots,a_n]$, for $n\geq 2$. Each non negative rational number admits an \em expansion in continuous fraction \em $[a_1,\ldots,a_n]$.
Each positive rational number admits two expansions. 
A positive integer $n$ admits the expansions  $[n]$ and $[n-1,1]$. If $\alpha\in\mathbb Q_{+}\setminus \mathbb Z$, $\al$ admits expansions  $[a_1,\ldots,a_n]$ and  $[a_1,\ldots,a_n-1,1]$, with  $a_n\ge 2$.
We set $[n,0]=+\infty$, for each non negative integer $n$.

\noindent
If $a_n\geq 1$, we define the \em length \em of $[a_1,\ldots,a_n]$ as $|[a_1,\ldots,a_n]|=a_1+\cdots+a_n$.
We set $|0|=|+\infty|=0$. 
Given   $\alpha=[a_1,\ldots,a_k]$, $\beta=[b_1,\ldots,b_\ell]\in\overline{\mathbb Q}_{+}$, we  say that $\alpha,\beta\in \overline{\mathbb Q}_{+}$ are \em related, \em  and write $\alpha\sim\beta$, if there is a non negative integer $n$ such that $\alpha=[b_1,\ldots,b_\ell,n]$ or $\beta=[a_1,\ldots,a_k,n]$. If  $\beta=[a_1,\ldots,a_k,n]$, we define the \em convolution of $\alpha$ and $\beta$ \em as the positive rational $\alpha\ast\beta=\beta\ast\alpha=[a_1,\ldots,a_k ,n+1]$.
\begin{equation}\label{CONVOLUTION}
\hbox{\rm  If $\alpha\sim\beta$,}   \qquad  \alpha\sim\alpha\ast\beta   \hbox{ \ \rm and } ~ |\alpha\ast\beta|=\max\{|\alpha|,|\beta|\}+1.
\end{equation}

\noindent
Given $\alpha,\beta,\gamma\in\overline{\mathbb Q}_{+}$, we say that $\beta$ \em is between $\alpha$ and $\gamma$ \em, and write $\alpha - \beta - \gamma$, if $\alpha<\beta<\gamma$ or $\gamma<\beta<\alpha$. Given $R\subset \overline{\mathbb Q}_{+}$ and $\alpha,\beta\in R$, we say that $\alpha$ and $\beta$ are \em contiguous \em in $R$, and write $\alpha R \beta$, if there is no $\gamma\in R$ such that $\alpha ~ - ~ \gamma ~ -  ~\beta$.

\noindent
Notice that $0\ast 1/n=[0]\ast [0,n]=[0,n+1]=1/(n+1)$, $n\ast +\infty=[n]\ast [n,0]=[n,1]=n+1$ and 
\begin{displaymath}
0 ~ - ~ \frac{1}{n+1} ~ - ~\frac{1}{n},\;\;n ~- ~n+1 ~- ~+\infty
\end{displaymath}

\noindent
Let  $\alpha,\beta,\gamma,\delta\in{\mathbb Q}_{+}$, $\delta=[a_1,\ldots,a_k]$, with $k\geq 1$. If $\alpha~ - ~ \beta ~ - ~ \gamma$,
\begin{displaymath}
\frac{1}{\alpha}~ - ~ \frac{1}{\beta} ~ - ~ \frac{1}{\gamma}\;\; \textrm{and}\;\;  \left(a_k+\frac{1}{\alpha}\right) ~ -  ~ \left(a_k+\frac{1}{\beta}\right)  ~ -  ~ \left(a_k+\frac{1}{\gamma}\right).
\end{displaymath}

\noindent
Hence $[\delta,\alpha]-[\delta,\beta]-[\delta,\gamma]$. Therefore $\alpha - \alpha\ast\beta - \beta$ whenever $\alpha\sim\beta$. Moreover,  $\alpha\sim\beta$ and $\alpha<\beta$ implies that
\begin{equation}\label{INEQ}
\alpha<\alpha\ast\beta<\beta.
\end{equation}
Relation (\ref{INEQ}) also holds if $\alpha=0$ or $\beta=+\infty$.

\noindent
Let $\alpha,\beta\in{\mathbb Q}_{+}$. We say that $\alpha$ \em precedes \em $\beta$, and we write $\alpha\prec\beta$, if $\beta=[a_1,\ldots,a_k]$ and there are integers $\ell,b$ such that $\alpha=[a_1,\ldots,a_{\ell},b]$ and $0\leq\ell\leq k-1,1\leq b\leq a_{\ell+1}$. The partial order $\prec$ induces in $\mathbb Q_{+}$ a structure of tree with root $1$. We say that $\al$ \em immediately precedes \em $\beta$ if $\al\prec\be$ and $|\be|=|\al|+1$.

\noindent
Given two divisors $E,F$ of a resolution graph $\mathcal R$, we will say that $E$ \em precedes $F$ \em 
and write $E\prec F$ if there is a sequence of divisors $E=E_0,E_1,\ldots,E_n=F$ 
such that $E_i$ is created blowing up a point of $E_{i-1}$, $1\le i \le n$.

\begin{lemma}\label{BIJECCAO}
There is a bijection $\alpha\mapsto E_{\alpha}$ from $\overline{\mathbb Q}_{+}$ onto $\Xi$ such that
\begin{itemize}
\item[$(a)$]
$\alpha<\beta$ if and only if $E_{\alpha}<E_{\beta}$.
\item[$(b)$]
If $\alpha\sim\beta$, $E_{\alpha}\ast E_{\beta}=E_{\alpha\ast\beta}$.
\item[$(c)$]
$\al\prec\be$ if and only if $E_\al\prec E_\be$.
\end{itemize}
\end{lemma}
\noindent
We will denote the inverse of $\alpha\mapsto E_{\alpha}$ by $E\mapsto \alpha^{E}$.

\begin{proof}
If $\mathcal R_{m}$ has vertices $E^{0}<E^{1}<\cdots<E^{r_{m}}$, 
$\mathcal R_{m+1}$ has vertices  $E^{0}<E^{0}\ast E^{1}<E^{1}<E^1\ast E^2<E^2\cdots<E^{r_{m}}$. Since $\mathcal R_{0}$ has two vertices, $\mathcal R_{m}$ has $2^{m}+1$ vertices, for each $m\geq 1$.

\noindent
Set $\overline{\mathbb Q}_{m}=\{\alpha\in \overline{\mathbb Q}_{+}:|\alpha|\leq m\}$ for each $m\ge 0$.
Each $\al\in\mathbb Q_+$ immediately precedes two elements of $\mathbb Q_+$: $1$ immediately precedes $2$ and $1/2$;
if $\al=[a_0,\ldots,a_k]$ and $a_k\ge 2$, $\al$ immediately precedes  $[a_0,\ldots,a_{k}+1]$, $[a_0,\ldots,a_k-1,2]$ and no other rational numbers. Hence $\#\{\al :|\al|=m\}=2^{m-1}$ and $\#\overline{\mathbb Q}_m=2^m+1$, for each $m\geq 1$.

\noindent
Let us show that, for each $m\ge 1$:

\noindent
$(1)_m$ $\overline{\mathbb Q}_{m}=\{\alpha_0,\ldots,\alpha_{2^{m}}\}$, $\alpha_0<\cdots<\alpha_{2^{m}}$, $|\alpha_i|=m$ if $i$ odd, $|\alpha_i|<m$ if $i$ even.

\noindent
$(2)_m$ If $\alpha,\beta\in\overline{\mathbb Q}_{m}$ and $\alpha\overline{\mathbb Q}_{m}\beta$, $\alpha\sim\beta$ .

\noindent
$(3)_m$ There is a bijection $\al\mapsto E_\al$ from $\overline{\mathbb Q}_{m}$ onto $\mathcal R_m$ such that $(a),(c)$ hold for $\al,\be\in\overline{\mathbb Q}_{m}$ and $(b)$ holds for $\al,\be\in\overline{\mathbb Q}_{m-1}$.

\noindent
Notice that $\overline{\mathbb Q}_{0}=\{0,+\infty\}$. If $|\alpha|=1$, $\alpha=1$. Hence $\overline{\mathbb Q}_{1}=\{0,1,+\infty\}$. Since  $+\infty=[0,0]$, $0\sim +\infty$ and $0\ast +\infty=[0,1]=1$.
By (\ref{CONVOLUTION}), $0\sim 1$ and $1\sim +\infty$. If $|\alpha|=2$, $\alpha=[0,2]=1/2$ or $\alpha=[1,1]=2$. Since $0\ast 1=[0]\ast [0,1]=[0,2]$ and $1\ast +\infty = [1]\ast [1,0]=[1,1]=2$, $0\sim 1/2$, $1/2 \sim 1$, $1 \sim 2$ and $2\sim +\infty$.

\noindent
Assume that $(1)_m$  holds for $i=1,2$.
By  (\ref{CONVOLUTION}) and (\ref{INEQ}),
\begin{displaymath}
|\alpha_{2i-2}\ast\alpha_{2i-1}|=|\alpha_{2i-1}\ast\alpha_{2i}|=m+1
\end{displaymath}

\noindent
and
\begin{equation}\label{DESIG}
\alpha_{2i-2}<\alpha_{2i-2}\ast\alpha_{2i-1}<\alpha_{2i-1}<\alpha_{2i-1}\ast\alpha_{2i}<\alpha_{2i},
\end{equation}

\noindent
for $1\leq i\leq 2^{m-1}$. Hence $(1)_{m+1}$ holds and $(2)_{m+1}$ follows from (\ref{CONVOLUTION}).

\noindent
Set $E_0=\Sigma$ and $E_{+\infty}=C$.
Assume that $(3)_m$ holds.
By $(b)$, $E$ has a unique extension to $\overline{\mathbb Q}_{m+1}$. By (\ref{DESIG}), $(a)$ holds for $\alpha,\beta\in\overline{\mathbb Q}_{m+1}$. By the definitions of $"\alpha\ast\beta"$ and $"E\ast F"$, $(b)$ holds for $\alpha,\beta\in\overline{\mathbb Q}_{m}$.

\noindent
Assume $\alpha\in\mathbb Q_{m}$, $|\beta|=m+1$.

\noindent
Assume $\alpha\prec\beta$. If $|\alpha|=m$, there is $\gamma\in\mathbb Q_{m}$ such that $\beta=\alpha\ast\gamma$. Hence $E_{\alpha}\prec E_{\beta}$. Otherwise, there is $\delta$ such that $|\delta|=m$ and $\alpha\prec\delta\prec\beta$. Hence $E_{\alpha}\prec E_{\delta}\prec E_{\beta}$.

\noindent
Assume $E_{\alpha}\prec E_{\beta}$. If $|\beta|=|\alpha|+1$, there is $\gamma\in\mathbb Q_{m}$ such that $E_{\beta}=E_{\alpha}\ast E_{\gamma}$. Hence $\beta=\alpha\ast\gamma$ and $\alpha\prec\beta$. Otherwise, there are $\gamma,\delta\in\mathbb Q_{m}$ such that $E_{\beta}=E_{\gamma}\ast E_{\delta}$ and $E_{\alpha}\prec E_{\gamma}$. Hence  $\alpha\prec\gamma\prec\gamma\ast\delta=\beta$.
\end{proof}

\section{Resolution Graphs of Cusp Type} \label{CSMOOTH}

\noindent
We keep the notations of Section \ref{TSMOOTH}. Let $Y$ be a germ of an irreducible plane curve of $(X,o)$, $Y\neq\Sigma$.
Assume that if $Y$ is singular, $Y$ has maximal contact with $C$ or  $Y$ has maximal contact with $\Sigma$. 
Assume $Y$ has at most one Puiseux exponent.
Given $\alpha> 1$ [$\alpha<1,\al=1$] we say that $Y$ is a \em cusp of exponent $\alpha$ relative to $(\Sigma,C)$ \em if 
$o(Y,C)=\alpha$ [$o(Y,\Sigma)=1/\alpha, Y$ is transversal to $C$ and $\Sigma$]. We assume $\alpha=p/q$, where $p,q$ are positive integers and $(p,q)=1$.

\noindent
Let $(x,y)$ be a system of local coordinates of $(X,o)$ such that $\Sigma=\{x=0\}$ and $C=\{y=0\}$. 
If $Y$ is a cusp of exponent $\alpha$ and $\alpha\geq 1$ [$\alpha<1$], there is $\varepsilon\in\mathbb C\{x^{1/q}\}$ [$\varepsilon\in\mathbb C\{y^{1/p}\}$] such that $Y=\{y=x^{\alpha}\varepsilon\}$ [ $Y=\{x=y^{\alpha}\varepsilon\}$] and $\varepsilon(0)\neq 0$. We say that $\varepsilon(0)$ is the \em coefficient of $Y$ relative to the system of local coordinates $(x,y)$. \em We say that $C$ has exponent $+\infty$ and coefficient $0$. Notice that the coefficient of a cusp depends on the system of local coordinates but the fact that two cusps have the same exponent but different coefficients does not.

\begin{lemma}\label{CUSP1}
Let $Y$ $[Y_i]$ be a cusp of exponent $\alpha$ relatively to $(\Sigma,C)$, $i=1,2$. 
Let $\widetilde{Y}$ $[\widetilde{Y}_i]$ be the strict transform of $Y$ $[Y_i]$ by the minimal resolution of $Y\cup \Sigma\cup C$ $[Y_i\cup \Sigma\cup C]$. Then
\begin{itemize}
\item[$(a)$]
$\widetilde{Y}$ is transversal to $E_{\alpha}$.
\item[$(b)$]
$\widetilde{Y}_1\cap \widetilde{Y}_2=\emptyset$ if and only if $\widetilde{Y}_1$ and $\widetilde{Y}_2$ have different coefficients.
\end{itemize}
\end{lemma}

\begin{proof}
By induction in $|\alpha|$.
\end{proof}

\begin{lemma}\label{CUSP2}
Let $(Y,\Sigma)$ be a logarithmic plane curve. Assume $\mathcal R^{\Sigma}_{Y}$ linear. Assume that $\widetilde{\Sigma}$ has valence $1$ and the terminal vertex of $\mathcal R^{\Sigma}_{Y}$ does not have weight $-1$. Then
\begin{itemize}
\item[$(a)$]
Each branch of $Y$ has at most one Puiseux exponent.
\item[$(b)$]
There is a smooth curve $C$ transversal to $\Sigma$ such that all singular branches of $Y$ have maximal contact with $C$ or $\Sigma$.
\item[$(c)$]
If $Y_i,Y_j$ are smooth branches of $Y$ transversal to $\Sigma$, $o(Y_i,Y_j)\leq \tau_{C}(Y)$.
\item[$(d)$]
All branches of $Y$ are cusps relatively to $(\Sigma,C)$. If two branches of $Y$ are cusps with the same exponent, they have different coefficients.
\end{itemize}
\end{lemma}

\begin{proof}
Statement $(a)$ is well known. See for instance \cite{BRIES}. Let us prove $(b)$. Let $\lambda_i$ be the exponent of each singular branch $Y_i$ of $Y$ transversal to $\Sigma$. If the curve $C$ does not exist, there are $i,j$ such that $o(Y_i,Y_j)<\min\{\lambda_i,\lambda_j\}$. After blowing up $Y_i\cap Y_j$ $o(Y_i,Y_j)$-times, the curves $Y_i$ and $Y_j$ are still singular and intersect different smooth points of $D_{o(Y_i,Y_j)}$. Hence $\mathcal R^{\Sigma}_{Y}$ is not linear.

\noindent
Assume that $(c)$ does not hold. Set $\tau=\tau_{C}(Y)$. Let $Y'$ [$\Sigma'$] be the germ of $Y^{(\tau)}$ [$E_{\tau}$] at $Y^{(\tau)}_i\cap Y^{(\tau)}_j$. By Lemma \ref{CUSP1}, the strict transforms of the singular branches of $Y$ only intersect divisors $E$ such that $\alpha^{E}<\tau$. Hence, the curve $(Y',\Sigma')$ verifies the conditions of Theorem \ref{SMOOTHLOG}. Therefore $\mathcal R^{\Sigma'}_{Y'}$ and $\mathcal R^{\Sigma}_{Y}$ have a terminal vertex of weight $-1$.

\noindent
Statement $(d)$ follows from Lemma \ref{CUSP1} and the linearity of $\mathcal R^{\Sigma}_{Y}$.
\end{proof}

\noindent
Let $(Y,\Sigma)$ be a logarithmic plane curve such that $\mathcal R^{\Sigma}_{Y}$ verifies the conditions of Lemma \ref{CUSP2}. Assume $Y_{i}$ is a branch of $Y$ such that $o(Y_i,C)>\tau_{C}(Y)$. By statement $(c)$, $Y_i$ is smooth and $o(Y_j,C)\leq \tau_{C}(Y)$, for all $j\neq i$. Hence we can take $C=Y_i$. Under these conditions, the set of divisors of $\mathcal R^{\Sigma}_{Y}$ is naturally identified with a subset of $\Xi$ (and a subset of $\overline{\mathbb Q}_{+}$).

\begin{lemma}\label{WEIGTH}
Assume $(Y,\Sigma)$ verifies the conditions of Lemma \ref{CUSP2}. Then:
\begin{itemize}
\item[$(a)$]
$E_{\alpha}\in \mathcal R^{\Sigma}_{Y}$ if and only if $\alpha$ belongs
to the smallest subtree $R$ of $(\mathbb Q_{+},\prec)$ that contains $1$ and
the exponents of the branches of $Y$.
\item[$(b)$]
If $\alpha=[a_1,\ldots,a_k]\in R$,
 \[
\omega_{E_{\alpha}}=-1-\max  \{c,0:[a_1,\ldots,a_k,c]\in R\}-\max
\{c,0:[a_1,\ldots,a_{k}-1,1,c]\in R\}.
\]
\end{itemize}
\end{lemma}

\begin{proof}
Let $Y$ [$Z$] be a cusp of type $[a_1,\ldots,a_k,c]$
[$[a_1,\ldots,a_{k}-1,1,d]$].
After blowing up $D_\ell\cap Y^{(\ell)}$, $\ell=0,1,\ldots,|\al|-1$,
one creates the divisor $E_\al$.
Moreover, the cusps $Y^{(|\al|)},Z^{(|\al|)}$ intersect $E_\al$ at
different points,
$o(Y^{(|\al|)},E_\al)=c$ and $o(Z^{(|\al|)},E_\al)=d$.

\noindent
After blowing up $D_\ell\cap Y^{(\ell)}$, $\ell=|\al|,\ldots,|\al|+c-1$ and
$D_\ell\cap Z^{(\ell)}$, $\ell=|\al|,\ldots,|\al|+d-1$,
$Y^{(|\al|+c)}$ and $Z^{(|\al|+d)}$ no longer intersect $E_\al$.
Moreover, $E_\al$ has weight $-1-c-d$.
\end{proof}

\noindent
Let $\mathcal R^{\Sigma}$ be a logarithmic resolution graph verifying the conditions of Lemma \ref{CUSP2}. We can give an explicit embedding of the set of vertices of $\mathcal R^{\Sigma}$ into $\mathbb Q_{+}\cup\{0\}$ by setting $\alpha^{\Sigma}=0$, $\alpha^{C}=+\infty$ and $\alpha^{E}=\alpha^{E'}\ast\alpha^{E''}$ whenever we are at a resolution step such that $\omega_{E}=-1$, $E'\cap E,E''\cap E\neq \emptyset$. Let $\tau(\mathcal R^{\Sigma})$ be the smallest positive integer $\tau$ such that for each non divisor $E$ of $\mathcal R^{\Sigma}$ with $\alpha^{E}\not\in \mathbb Z$, $\alpha^{E}<\tau$.

\noindent
Assume that the resolution graph $\mathcal R_{Y}$ of $Y$  is linear, its root has valence $1$ and its vertex does not have weight $-1$. Let $\Sigma$ be a smooth curve transversal to $Y$. We obtain $\mathcal R^{\Sigma}_{Y}$ from $\mathcal R_Y$ connecting $\Sigma$ to the root of $\mathcal R_Y$ and setting $\omega_{\Sigma}=-2$. By Lemma \ref{CUSP2}, the set of divisors of $R_Y$ is naturally identified with a subset of $\mathbb Q_{+}$. Set $\tau(\mathcal R_{Y})=\tau(\mathcal R^{\Sigma}_{Y})$.

\em\begin{definition}\label{LOGRAPH}
Let $\mathcal R$ [$\mathcal R^{\Sigma}$] be a [logarithmic\/] resolution graph. 
We say that $\mathcal R$ [$\mathcal R^{\Sigma}$] is a   [\em log\/\em ]\em  elementary graph of cusp type \em if:
\begin{itemize}
\item[$(a)$]
$\mathcal R$ [$\mathcal R^{\Sigma}$] is a linear graph and $\Sigma$ has valence $1$.
\item[$(b)$]
$\mathcal R$ [$\mathcal R^{\Sigma}$] is not a [log] elementary graph of smooth type.
\item[$(c)$]
If an integer divisor $E$ of $\mathcal R$ [$\mathcal R^{\Sigma}$] intersects a strict transform of a branch of the logarithmic curve, $\alpha^{E}=\tau(\mathcal R)$  [$\alpha^{E}=\tau(\mathcal R^{\Sigma})$].
\end{itemize}
\end{definition}\em

\noindent
Condition $(c)$ guarantees the uniqueness of the decomposition (see Definition \ref{GLUE} and Examples \ref{EXE3}, \ref{EXE4}).


\begin{theorem}\label{LINEARLOG}
Let $(Y,\Sigma)$ be a logarithmic plane curve. Its resolution graph is a log elementary graph of cusp type if and only if:
\begin{itemize}
\item[$(a)$]
there is a smooth curve $C$ transversal to $\Sigma$ and a system of local coordinates $(x,y)$ such that $\Sigma=\{x=0\}$, $C=\{y=0\}$ and all branches of $Y$ are cusps relatively to $(\Sigma,C)$.
\item[$(b)$]
If two branches of $Y$ have the same exponent, they have different coefficients.
\item[$(c)$]
If $Y_i$ is a branch with integer exponent $\alpha$, $\alpha=\tau_{C}(Y)$.
\item[$(d)$]
$Y$ has a singular branch or a branch tangent to $\Sigma$.
\end{itemize}
\end{theorem}

\begin{proof}
Assume $\mathcal R^{\Sigma}_{Y}$ is a log elementary graph of cusp type.
Statement $(d)$ follows from condition $(b)$ (of Definition \ref{LOGRAPH}) and Theorem \ref{SMOOTHLOG}.
Statements $(a),(b)$ follow from condition $(a)$ and Lemma \ref{CUSP2}.
Assume that $Y$ has several smooth branches transversal to $\Sigma$.
By Lemma \ref{CUSP2} its exponents are smaller than or equal to $\tau_C(Y)$.
By condition $(c)$, these exponents are equal to $\tau_C(Y)$.

\noindent
Let us prove the converse.
Condition $(a)$ follows from statements $(a),(b)$ and Lemmas \ref{BIJECCAO},\ref{CUSP1}.
Condition $(b)$ follows from statement $(d)$ and Lemma \ref{WEIGTH}. 
Condition $(c)$ follows from statement $(c)$.
\end{proof}

\noindent
We will depict the vertices of an elementary graph of cusp type using black dots.

\begin{example}\label{EXE2}
\em An elementary graph of cusp type. \em
Let $Y_i=\{ y=\varphi_i (x)\}$, $1\le i\le 3$, where 
$\varphi_1(x)=x^{3/2}$,
$\varphi_2(x)=x^{5/3}$ and 
$\varphi_3(x)=x^{5/2}$. Set $Y=\cup_{i=1}^{3}Y_i$. 

\begin{figure}[h!]
\centering
\includegraphics{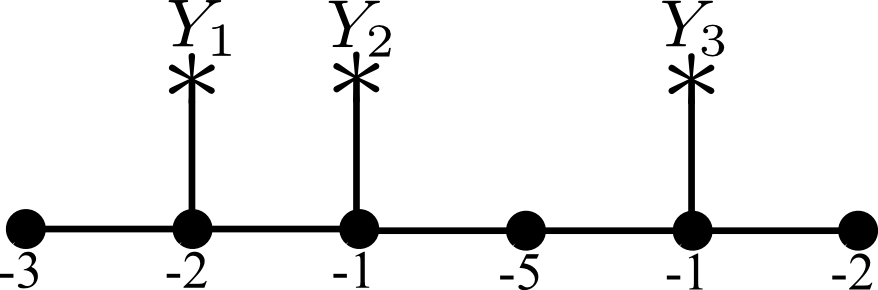}
\label{FIGEXE2}
\caption{The elementary graph $\mathcal R_{Y}$.}
\end{figure}

\end{example}

\begin{corollary}\label{LINEAR}
Let $Y$ be the germ of a plane curve. Its resolution graph is a elementary graph of cusp type if and only if:
\begin{itemize}
\item[$(a)$]
There is a smooth curve $\Sigma$ transversal to $Y$ and there is a smooth curve $C$ transversal to $\Sigma$ such that all branches of $Y$ are cusps relatively to $(\Sigma,C)$.
\item[$(b)$]
If two branches of $Y$ have the same exponent, they have different coefficients.
\item[$(c)$]
If $Y_i$ is a branch with integer exponent $\alpha$, $\alpha=\tau_{C}(Y)$.
\item[$(d)$]
$Y$ has a singular branch.
\end{itemize}
\end{corollary}

\section{Decomposition of a Resolution Graph}\label{DECOM}

\begin{definition}\label{GLUE}
Let $\mathcal G$ be an elementary graph, $E$ a vertex of $\mathcal G$ and $Z$ a branch that intersects $E$. 
Let $\mathcal H$ be a log elementary graph with root $\Sigma$. We glue $\mathcal G$
with $\mathcal H$ at $Z$ withdrawing $Z$, identifying $\Sigma$ with $E$ and replacing $\omega_E$ by $\omega_E+\omega_{\Sigma}+1$. 
The gluing of $\mathcal G$ with $\mathcal H$ at $Z$ is \em admissible \em  if $E$ is not a black root and one of the following conditions is verified:
\begin{enumerate}
\item[(a)]
$E$ and $\Sigma$ have different colors.
\item[(b)]
$E$ is black and the number of divisors connected to $E$ plus the number of branches connected to $E$ is greater than $2$.
\end{enumerate}
Let $\Lambda$ be a rooted tree with root $\phi$. Let $(\mathcal D_{u})_{u\in\Lambda}$ be a family of elementary graphs. Assume all log elementary with the possible exception of $\mathcal D_{\phi}$. Associate to each vertex $u$ of $\Lambda$, $u\not=\phi$, its father $f(u)$, a divisor $E_u$ of $\mathcal D_{f(u)}$ and a branch $Y_u$ that intersects $E_u$. Assume that the map $u\mapsto Y_u$ is injective.
We call \em gluing of \em $(\mathcal D_u)_{u\in \Lambda}$ to the graph we obtain gluing $\mathcal D_{f(u)}$ with $\mathcal D_u$ at $Y_u$ for each $u\in\Lambda\setminus\{\phi\}$.

\noindent
We say that a resolution graph \em admits a decomposition
into elementary graphs \em if it is a gluing of elementary graphs.
A decomposition into elementary graphs of a resolution graph is called \em admissible \em if, the gluing of $\mathcal D_{f(u)}$ with $\mathcal D_u$ at $Y_u$ is admissible  for each $u\in \Lambda\setminus \{ \phi \}$.
\end{definition}

\begin{theorem}\label{GLUEING}
A gluing of elementary graphs is a resolution graph.
\end{theorem}

\begin{proof}
Let us prove the theorem by induction on the number of vertices of $\Lambda$.
The theorem holds if  $\Lambda=\{\phi\}$.
Let $u$ be a terminal vertex of $\Lambda$, $u\not=\phi$. Set $\Lambda_{0}=\Lambda\setminus \{u\}$ and let $o$ be the intersection point of $Y_u$ and $E_u$. 
By the induction hypothesis, the gluing of $(\mathcal D_v)_{v\in \Lambda_{0}}$ 
is the resolution graph of a plane curve $H$. 
Let $\pi:\widetilde{X}\rightarrow (X,\sigma)$ be the resolution of $H$. 
Set ${D}=\pi^{-1}(\sigma)$. Let $\widetilde{H}$ be the strict transform of $H$ by $\pi$.  
Let $\widetilde{H}_u$ be the germ $\widetilde{H}$ at $o$. 
By Theorems \ref{LINEARLOG}, \ref{SMOOTHLOG} there is a germ of a logarithmic curve $(W,\Sigma)$ with resolution graph $\mathcal D_{u}$.
We can assume that $\Sigma$ is the germ of $E_u$ at $o$.
Set $Y=\pi((\widetilde{H}\setminus H_u)\cup W)$.
The resolution graph of $Y$ is obtained gluing the resolution graph of $H$ and the resolution graph of 
$(W,\Sigma)$ through the identification of $\Sigma$ with a germ of $E_u$.
Let $\widetilde{E}$ be the strict transform of $E_u$ 
by the sequence of blow ups that desingularizes $(\widetilde{H}\setminus H_u)\cup W$.
If $\kappa$ denotes the number of blow ups necessary to separate $W$ and $\Sigma$, 
$\om_{\widetilde E}$ equals $\omega_{E_{u}}-\kappa$.
\end{proof}

\noindent 
We call \em leading term \em of a Puiseux series $\sum_{\alpha}a_{\alpha}x^{\alpha}$ to
$a_{\beta} x^{\beta}$, where $\beta$ is the smallest positive rational $\gamma\in \mathbb Q$ such that $a_{\gamma}\neq 0$. If $Z=\{y=\sum_\al a_\al x^\al \}$ is a germ of plane curve, set $Z^{[\ell]}=\{y=\sum_{\al \leq\ell}a_\al x^j\}$. 

\begin{lemma}\label{LEMAPRINCIPAL}
Let $(Y,\Sigma)$ be  a logarithmic plane curve, with $\Sigma$ smooth.
There is a curve $Y^{\phi}$ such that:
\begin{enumerate}
\item[(a)]  the Puiseux expansion of each irreducible component $Y^{\phi}_i$ of $Y^{\phi}$ is a truncation of the Puiseux expansion of some irreducible component $Y_i$ of $Y$;

\item[(b)]  the resolution graph of  $(Y^{\phi},\Sigma)$ is a log elementary graph;

\item[(c)]  the sequence of blow ups that desingularize $(Y,\Sigma)$ and $(Y^\phi,\Sigma)$ coincide up to the step when the inverse image of $Y^\phi\cup\Sigma$ is normal crossings; 

\item[(d)]  
let $\pi_{\phi}:X^\phi\to X$ \em [$\pi:\widetilde{X}\to X$] \em be the sequence of blow ups that desingularizes $(Y^\phi,\Sigma)$ \em [$(Y,\Sigma)$]. \em Let $\widetilde Y$ be the strict transform of $Y$ by $\pi_{\phi}$.
Set $D^\phi=\pi^{-1}_{\phi}(\Sigma\cap Y)$. 
Let $\sigma$ be a point of $D^{\phi}$ such that $D^{\phi}\cup \widetilde{Y}$ is not normal crossings at $\si$.
Let $D^\phi_\si$ \em [$\widetilde{Y}_\si$] \em be the germ of $D^\phi$ \em [$\widetilde{Y}]$ \em at $\sigma$. 
Then there is a truncation $\widetilde{Y}^{\phi}_\si$ of $\widetilde{Y}_\si$ such that the resolution graph of $(\widetilde{Y}^{\phi}_\si,D^\phi_\si)$ is an elementary graph and the gluing of $\mathcal R^{\Sigma}_{Y^\phi}$ with 
$\mathcal R^{D^\phi_\si}_{\widetilde{Y}^{\phi}_\si}$ at $\widetilde Y^\phi_\si$ is admissible.
 \em
\end{enumerate}
\end{lemma}

\begin{proof}
Set $J=\{i\in I: Y_i$ is transversal to $\Sigma\}$ and $K=I\setminus J$. 
For $i\in J$, set $\lambda_i$ as the first Puiseux exponent of $Y_i$, if $Y_i$ is singular, and $\lambda_i=+\infty$ otherwise.
If $Y_i$ is smooth for each $i\in J$, let $(x,y)$ be a system of local coordinates such that $\Sigma=\{x=0\}$.
Assume that there is $i\in J$ such that $Y_i$ is singular.
Let $\rho$ be the supremum of the set of integers $\ell$ such that there is a smooth curve $C$ that verifies, for each $i\in J$, $o(Y_i,C)\ge \ell$ or $o(Y_i,C)=\lambda_i$. Choose a smooth curve $C$ such that $C$ is transversal to $\Sigma$ and, for each $i\in J$, $o(Y_i,C)\ge \rho$ or $o(Y_i,C)=\lambda_i$.
Let $(x,y)$ be a system of local coordinates such that $\Sigma=\{x=0\}$ and $C=\{y=0\}$.
Let $\tau_0$ be the supremum  of  $\lambda_i$ such that $i\in J$, $Y_i$ is singular and $\lambda_i\le\rho$. 
Let $\tau$ be the smallest positive integer bigger than $\tau_0$. 

\noindent
$(A)$ \em Assume $K=\emptyset$ and $\tau=1$. \em  Let $q_i$ be the supremum of the set of integers $\ell$ such that $Y^{[\ell]}_{i}$ is smooth and $o(Y^{[\ell]}_{i},Y^{[\ell]}_{j})\in \mathbb Z\cup \{+\infty\}$ for each $j\in I$. Set $Y^{\phi}_{i}=Y^{[q_{i}]}_i$ for each $i\in I$. Set $Y^{\phi}=\cup_{i\in I}Y^{\phi}_{i}$.

\noindent
$(B)$ \em Assume $K\neq \emptyset$ or $\tau\geq 2$. \em 
If $i\in K$ and $Y_i=\{x=\psi_{i}(y)\}$, set $Y^{\phi}_{i}=\{x=\widetilde{\psi}_{i}(y)\}$, where  
$\widetilde{\psi}_i$ is the leading term of $\psi_{i}$. 
Assume $i\in J$ and  $Y_i=\{y=\varphi_{i}(x)\}$. If $o(Y_i,C)>\tau$, set $Y^{\phi}_{i}=C$. Otherwise, set $Y^{\phi}_{i}=\{y=\widetilde{\varphi}_{i}(x)\}$, where  
$\widetilde{\varphi}_i$ is the leading term of $\varphi_{i}$. Set $Y^{\phi}=\cup_{i\in I}Y^{\phi}_{i}$.

\noindent
By construction, $(Y^{\phi},\Sigma)$ verifies condition $(a)$. In case $(A)$ [$(B)$], $(Y^{\phi},\Sigma)$ verifies the conditions of Theorem \ref{SMOOTHLOG}  [Theorem  \ref{LINEARLOG} ], hence statement $(b)$ holds.

\noindent 
Statement $(c)$ follows from Theorems \ref{BLOWINV1} and \ref{BLOWINV2}.

\noindent
Let us show that statement $(d)$ holds.

\noindent
\em Assume $(A)$. \em Let $\widetilde{Y_i}$ $[\widetilde{Y^{\phi}_{i}}]$ be the strict transform of $Y_i$ $[Y^{\phi}_{i}]$ by $\pi_{\phi}$. Choose $\sigma$ such that $D^{\phi}\cap \widetilde{Y}$ is not normal crossings at $\sigma$. By the definition of the $q_i$'s, there is $i\in I$ such that $\sigma\in \widetilde{Y_{i}}$, $Y^{[q_i+1]}_i$ is singular and $\widetilde{Y_{i}}$ is either singular or tangent to $D^{\phi}$ at $\sigma$. Hence $\widetilde{Y^{\phi}_{\sigma}}$ verifies condition $(B)$.

\noindent
\em Assume $(B)$. \em Let $\widetilde{C}$ be the strict transform of $C$ by $\pi_{\phi}$. By the definition of admissibility, we only need to check what happens at the point $\sigma$ such that $\widetilde{C}\cap D^{\phi}=\{\sigma\}$. Set $J_\phi=\{ i\in J: Y^\phi_i=C\}$, $Y_\phi=\cup_{i\in J_\phi}Y_i$.
Notice that the germ at $\sigma$ of the strict transform of $Y$ by $\pi_{\phi}$  equals  the strict transform of $Y_\phi$ by $\pi_\phi$. If $J_\phi=\emptyset$ or $Y_\phi$ is smooth, nothing is glued to $\mathcal R^{\Sigma}_{Y^\phi}$ at $\sigma$ and there is nothing to prove. Assume $Y_\phi$ is singular. By the construction of $Y^\phi$, in particular the definitions of $\rho$ and $\tau$, the minimum of the set $\{ o(Y_i,Y_j):~ i,j\in J_\phi\}$ is an integer.
Hence $\widetilde{Y}^{\phi}_{\sigma}$ verifies condition $(A)$.
\end{proof}

\begin{theorem}\label{LOGDECOMPOSITION}
The resolution graph of a logarithmic plane curve $(Y,\Sigma)$, with $\Sigma$ smooth, admits one and only one admissible decomposition into elementary graphs.
\end{theorem}

\begin{proof}
Let $g_i$ be the number of Puiseux pairs of the branch $Y_i$ of $Y$, $i\in I$. Set $h_i=g_i$ if $Y_i$ is transversal to $\Sigma$. Otherwise, set $h_i=g_i+1/2$.
Set $h_Y=\max\{h_i:i\in I\}$. 
Let $\#_Y$ be the cardinality of $I$.
Let us construct an admissible decomposition into elementary graphs  $(\mathcal D_v)_{v\in\Lambda}$
of $\mathcal R^{\Sigma}_{Y}$ by induction on $h_Y$ and $\#_Y$.

\noindent
Let $Y^\phi$ be the curve constructed in Lemma \ref{LEMAPRINCIPAL}. We will follow the notations of this Lemma.

\noindent
If $\si\in D^{\phi}$ and $D^{\phi}\cup\widetilde{Y}$ is not normal crossings at $\si$, $h_{\widetilde{Y}_{\si}}<h_{Y}$ or $\#_{\widetilde{Y}_{\si}}<\#_{Y}$. Hence $(\widetilde{Y}_{\si},D^{\phi}_{\si})$ admits an admissible decomposition into log elementary graphs. Since $\mathcal R^{\Sigma}_{Y}$ is the gluing of $\mathcal R^{\Sigma}_{Y^{\phi}}$ and the graphs $\mathcal R^{D^\phi_\si}_{\widetilde{Y}^{\phi}_\si}$, we obtain an admissible decomposition into elementary graphs of  $\mathcal R^{\Sigma}_{Y}$.

\noindent
Let $(\mathcal E_v)_{v\in\Lambda'}$ be another admissible decomposition of $\mathcal R^{\Sigma}_{Y}$.
Let us show by induction on $h_Y$ and $\#_Y$ that the two decompositions are equal.
It is enough to show that $\mathcal D_\phi=\mathcal E_\phi$.


\noindent
\em Assume that $\mathcal D_{\phi}$ is of cusp type and $\mathcal E_{\phi}$ is of smooth type. \em Let $(E,\Sigma)$ [$(D,\Sigma)$] be a logarithmic plane curve such that $\mathcal R^{\Sigma}_{E}=\mathcal E_{\phi}$ [$\mathcal R^{\Sigma}_{D}=\mathcal D_{\phi}$]. let $\pi_{E}:X^{E}\to X$ [$\pi_{D}:X^{D}\to X$] be the sequence of blow ups that desingularizes $(E,\Sigma)$ [$(D,\Sigma)$]. There is a subtree $R$ of $(\mathbb Q_{+},\prec)$ such that we can identify the set of vertices of $\mathcal D_{\phi}$ with $\{E_{\alpha}:\alpha\in R\}$. Let $m$ be the biggest integer of $R$. By condition $(c)$ of Definition \ref{LOGRAPH}, the vertices $E_1,\ldots,E_m$ of $\mathcal D_{\phi}$ are also vertices of $\mathcal E_{\phi}$ and the strict transform of $E$ by $\pi_{D}$ intersects $E_m$ at a smooth point of $\pi^{-1}_{D}(\Sigma)$. Hence, The strict transform of $D$ by $\pi_{E}$ intersects $E_m$ at a singular point of $\pi^{-1}_{E}(\Sigma)$. Therefore the decomposition $(\mathcal E_v)_{v\in\Lambda'}$ is not admissible and $\mathcal D_{\phi}$ and $\mathcal E_{\phi}$ must be of the same type. Let $G$ be a vertex of $\mathcal E_{\phi}$ and $F$ a vertex of $\mathcal R^{\Sigma}_{Y}$. Since $\mathcal E_{\phi}$ is a resolution graph,
\begin{equation}\label{TRANSITION}
\textrm{if}\; F\prec G,\; F\; \textrm{is a vertex of}\; \mathcal E_{\phi}.
\end{equation}


\noindent
\em Assume that there is a vertex of $\mathcal D_{\phi}$ that is not a vertex of $\mathcal E_{\phi}$. \em Let $\gamma$ be the biggest element of $R$ such that $E_{\gamma}$ is a vertex of $\mathcal E_{\phi}$. Let $n$ be the smallest integer bigger than or equal to $\gamma$. Then $\gamma=n$. Otherwise $E_{n}\prec E_{\gamma}$ and $\mathcal E_{\phi}$ would not be a resolution graph by (\ref{TRANSITION}). Since $\mathcal D_\phi$ is linear and $\mathcal D_\phi$ properly contains $\mathcal E_\phi$, $E$ has a smooth branch $E'$ such that the strict transform of $E'$ by $\pi_{E}$ intersects $E_n$. Therefore $(D,\Sigma)$ does not verify condition $(c)$ of Definition \ref{LOGRAPH}. Hence the decomposition $(\mathcal E_v)_{v\in\Lambda'}$ is not admissible.

\noindent
If there is a vertex $E$ of $\mathcal E_{\phi}$ that is not a vertex of $\mathcal D_{\phi}$, the same argument shows that $(\mathcal D_v)_{v\in\Lambda}$ is not admissible.

\noindent
\em Assume that $\mathcal D_{\phi}$ is of smooth type. \em Hence $\mathcal E_{\phi}$ is also of smooth type. Assume $\mathcal D_{\phi}\neq \mathcal E_{\phi}$. Hence we can assume that there is a vertex $E$ of  $\mathcal D_{\phi}$ and 
$\mathcal E_{\phi}$ is connected to a vertex $F$ of $\mathcal E_{\phi}$  that is not a vertex of $\mathcal D_{\phi}$. Let $\mathcal F$ be the subtree of $\mathcal E_{\phi}$ with vertices the divisors $D$ of $\mathcal E_{\phi}$ such that $D=F$ or $F\prec D$. Let us reset $\omega_{E}$ in such a way that $\omega_{E}$ plus the set of vertices of $\mathcal F$ connected to $E$ equal $-1$. We obtain in this way a resolution graph $(\mathcal F,E)$. Since $\mathcal E_{\phi}$ is of smooth type, $\mathcal F$ is of smooth type. Hence the gluing of $(\mathcal D_v)_{v\in\Lambda}$ is not admissible.
\end{proof}

\begin{corollary}\label{DECOMPOSITION}
The resolution graph of a plane curve admits one and only one admissible decomposition into elementary graphs.
\end{corollary}

%

\begin{example}\label{EXE3}

Let $Y_i=\{ y=\varphi_i (x)\}$, $i=1,2$, where 
$\varphi_1(x)=x^2+x^{7/2}$ and 
$\varphi_2(x)=x^{7/2}$. We have $Y^{\phi}_{i}=\{y=\varphi^{\phi}_i(x)\}$, $i=1,2$, with $\varphi^{\phi}_1(x)=x^2$ and $\varphi^{\phi}_2(x)=0$. Set $Y=Y_{1}\cup Y_{2}$. See Figure \ref{F3}. The decomposition (C) is not admissible since the resolution graph of $Y_1\cup Y^{\phi}_{2}$ does not verify condition (c) of Definition \ref{LOGRAPH}. 

\begin{figure}[h]
\centering
\subfloat[The graph $\mathcal R_{Y}$]{
\includegraphics{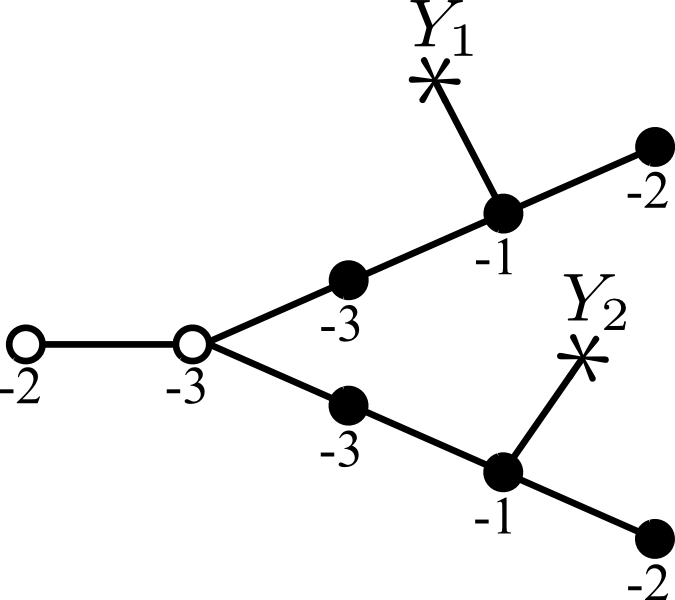}
\label{FIGEXE3}}
\qquad \qquad
\subfloat[The canonical decomposition of $\mathcal R_{Y}$]{
\includegraphics{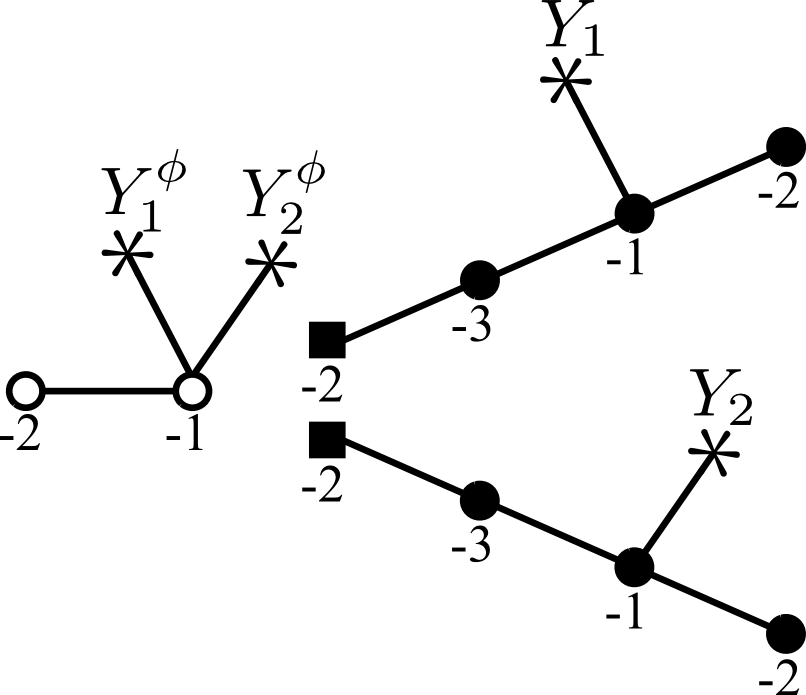}
\label{FIGEXE3D}}\\
\subfloat[Another decomposition of $\mathcal R_{Y}$]{
\includegraphics{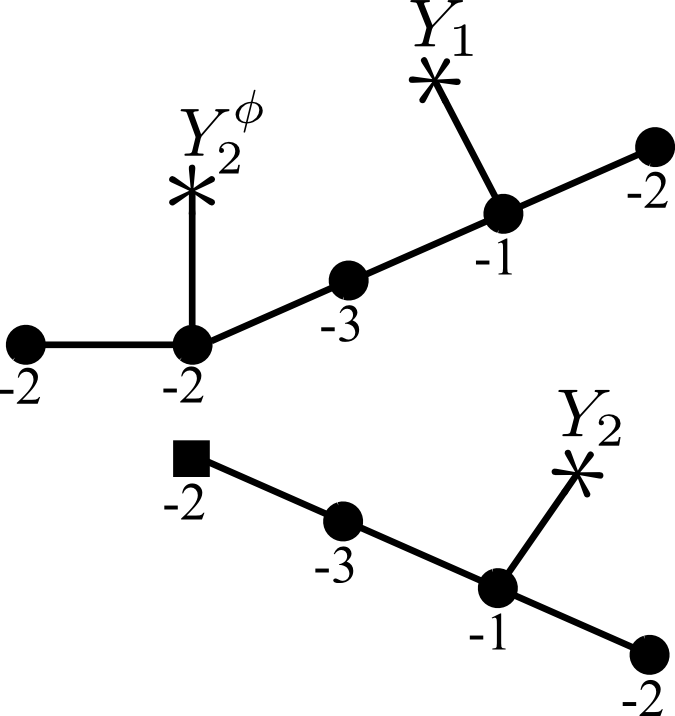}
\label{FIGEXE3A1}}
\qquad \qquad
\caption{Example \ref{EXE3}. \label{F3}}
\end{figure}

\end{example}

\begin{example}\label{EXE4}

Let $Y_i=\{ y=\varphi_i (x)\}$, $1\le i\le 4$, where 
$\varphi_1(x)=x^{3/2}$, 
$\varphi_2(x)=x^{5/2}$,
$\varphi_3(x)=x^{2}+x^{4}$ and 
$\varphi_4(x)=x^{2}-x^{4}$.  We have $Y^{\phi}_{i}=\{y=\varphi^{\phi}_i(x)\}$, $1\geq i\geq 4$, with  $\varphi^{\phi}_1(x)=\varphi_1(x)=x^{3/2}$, $\varphi^{\phi}_2(x)=0$ and $\varphi^{\phi}_3(x)=\varphi^{\phi}_4(x)=x^2$.
Set $Y=\cup_{i=1}^{4}Y_i$. See Figure \ref{F4}. The decomposition (C) is not admissible since the resolution graph of $Y_1\cup Y^{\phi}_{3}\cup Y_2$ does not verify condition (c) of Definition \ref{LOGRAPH}.

\begin{figure}[h]
\centering
\subfloat[$\mathcal R_{Y}$]{
\includegraphics{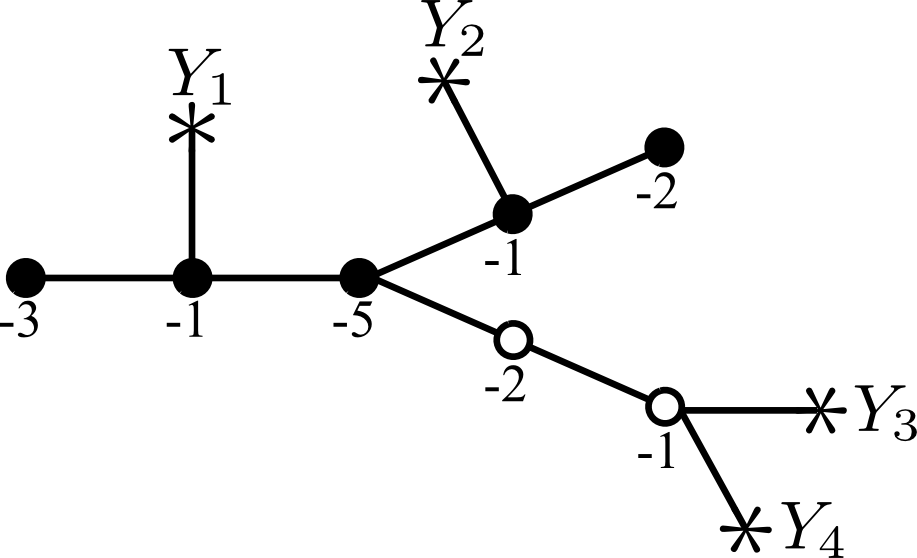}
\label{FIGEXE4}}
\qquad \qquad
\subfloat[The canonical decomposition of $\mathcal R_{Y}$]{
\includegraphics{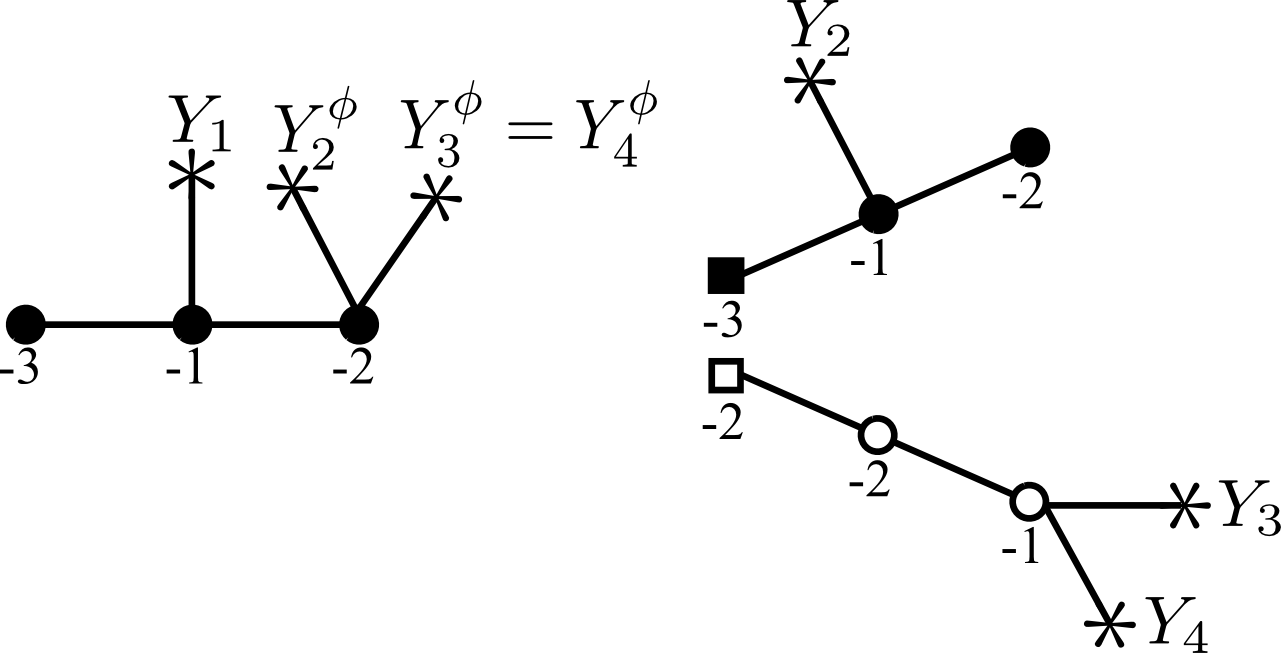}
\label{FIGEXE4D}}\\
\subfloat[Another decomposition of $\mathcal R_{Y}$]{
\includegraphics{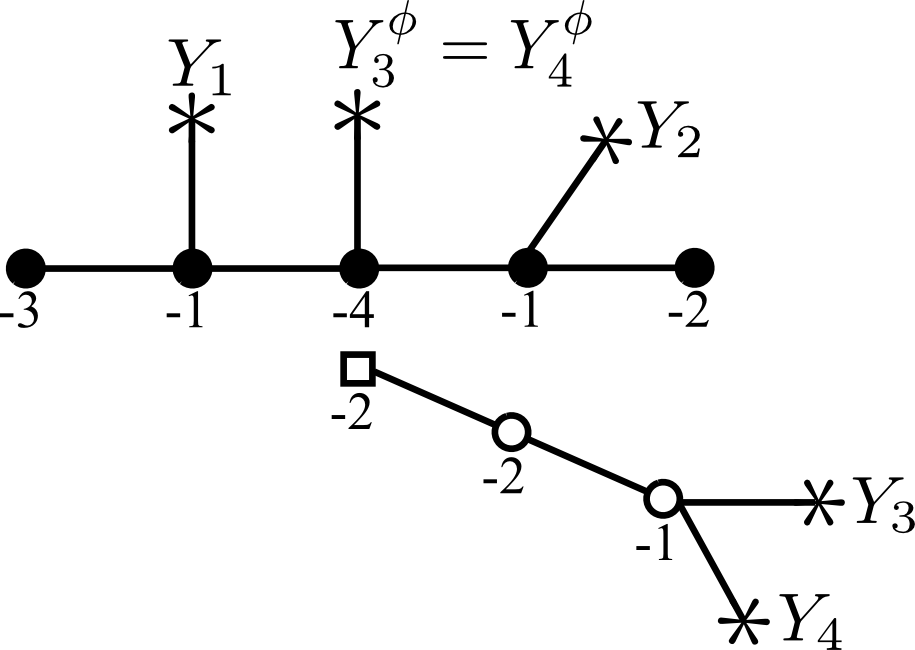}
\label{FIGEXE4A1}}
\caption{Example \ref{EXE4}. \label{F4}}
\end{figure}

\end{example}

\bibliographystyle{amsplain}

\newpage
\noindent
Jo\~ao Cabral  CMAF, Universidade de Lisboa and Departamento de Matem\'atica, Faculdade de Ci\^{e}ncias e Tecnologia, Universidade Nova de Lisboa \\
\email{joao.cabral.70@gmail.com}

\medskip\noindent
Orlando Neto  CMAF and Faculdade de Ci\^{e}ncias, Universidade de Lisboa  \\
\email{orlando60@gmail.com}

\medskip\noindent
Pedro C. Silva   Centro de Estudos Florestais and Departamento de Ci\^encias e Engenharia
de Biossistemas, Inst. Superior de Agronomia, Universidade de Lisboa \\
\email{pcsilva@isa.utl.pt}
 
 \medskip\noindent
This research was partially supported by FEDER and FCT-Plurianual 2011.

\end{document}